\documentclass{amsart}
\usepackage{amsfonts}
\usepackage{amsmath,amssymb}
\usepackage{amsthm}
\usepackage{amscd}
\usepackage{graphics}
\usepackage{graphicx}

\theoremstyle{remark}{
\newtheorem{Def}{{\rm Definition}}
\newtheorem{Ex}{{\rm Example}}
\newtheorem{Rem}{{\rm Remark}}

}
\theoremstyle{plain}
{

\newtheorem{Thm}{Theorem}

}

\begin{document}
\title[Real algebraic curve-valued functions with prescribed Reeb graphs]{Reconstruction of real algebraic functions into curves with prescribed Reeb graphs}
\author{Naoki kitazawa}

\address{Osaka Central Advanced Mathematical Institute (OCAMI) \\
3-3-138 Sugimoto, Sumiyoshi-ku Osaka 558-8585
TEL: +81-6-6605-3103
}
\email{naokikitazawa.formath@gmail.com}
\urladdr{https://naokikitazawa.github.io/NaokiKitazawa.html}
\maketitle
\begin{abstract}
We discuss reconstructing smooth real algebraic maps onto curves whose {\it Reeb graph} is as prescribed. 
The {\it Reeb graph} of a smooth function is the space of all connected components of preimages of all single points and a natural quotient space of the manifold with the vertex set being all connected components containing some singular points of it. This gives a strong tool in geometry of manifolds and appeared already in 1950 with Morse functions. 

The Reeb graph of the natural height of the unit sphere of dimension at least $2$ is a graph with exactly two vertices and one edge. We reconstruct functions, from general finite graphs, conversely. In the differentiable situations, Sharko pioneered this in 2006, followed by Masumoto-Saeki and Michalak, mainly. Related real algebraic situations have been launched and studied by the author. The curve-valued case is first considered here.

\end{abstract}
\section{Introduction: history on our study and terminologies, notions and notation we need and our main result.}
\label{sec:1}
This paper is on construction of explicit real functions with prescribed topological and combinatorial properties.
Constructing the functions is different from knowing the existence of such functions. Existence theory (and approximation) on real algebraic manifolds and maps are well-known as a kind of classical theory and related real algebraic geometry has been founded by Nash and Tognoli \cite{nash, tognoli}, and is developing. See \cite{kollar} for related history, for example.

Systematic construction is difficult in general, and important.
We are concerned with systematic construction of real algebraic functions which are regarded as generalized versions of the canonical projections of the unit spheres and their topological and combinatorial properties. More precisely, we are interested in the Reeb graphs of smooth real algebraic functions. The {\it Reeb graph} of a smooth function is the space of all connected components of preimages of all single points and a quotient space of the manifold with the vertex set being all connected components containing some singular points of the function. The height function of the unit sphere is of simplest smooth real algebraic functions and its Reeb graph is a graph with exactly two vertices and one edge. We are concerned with reconstruction of real algebraic functions whose Reeb graphs are as prescribed. Related studies in the differentiable situations are pioneered by \cite{sharko}, followed by \cite{masumotosaeki}, and \cite{michalak}, mainly. The author has also contributed to this (e.g. \cite{kitazawa1}). Related real algebraic studies are due to the author. The curve-valued function case is first studied here. 

We review some fundamental and important terminologies, notions and notation on manifolds and graphs, rigorously.

\subsection{On smooth or real algebraic manifolds and maps and graphs.}
\subsubsection{Smooth manifolds and maps.}
Let ${\mathbb{R}}^k$ denote the $k$-dimensional Euclidean space. It is also the $k$-dimensional real affine space. It is also a Riemannian manifold endowed with the so-called standard Euclidean metric. For a point $x \in {\mathbb{R}}^k$, let $x_j$ denote the $j$-th component, for an integer $1 \leq j \leq k$, hereafter. For two points $x_1,x_2 \in {\mathbb{R}}^k$, let $||x_1-x_2||:=\sqrt{{\Sigma}_{j=1}^k {(x_{1,j}-x_{2,j})}^2}$ denote the distance of the two points induced by the metric. We also use $||x||:=||x-0||$ in the case $0 \in {\mathbb{R}}^k$ is the origin. Let $D^k:=\{x \in {\mathbb{R}}^k \mid ||x|| \leq 1\}$, the $k$-dimensional unit disk, and $S^{k-1}:=\{x \in {\mathbb{R}}^k \mid ||x||=1\}$, the $k$-dimensional unit sphere.
Let ${\pi}_{m,n}:{\mathbb{R}}^m \rightarrow {\mathbb{R}}^n$ denote the canonical projection, mapping $x=(x_1,x_2) \in {\mathbb{R}}^n \times {\mathbb{R}}^{m-n}={\mathbb{R}}^m$ to $x_1$ with $m>n \geq 1$. It is of simplest real polynomial maps: a real polynomial map $c:{\mathbb{R}}^m \rightarrow {\mathbb{R}}^n$ is a map each component $c_j:{\mathbb{R}}^m \rightarrow \mathbb{R}$ (the $j$-th component) is represented by a real polynomial where $m$ and $n$ are arbitrary positive integers. 

For a differentiable manifold $X$, let $T_x X$ denote the tangent vector space of $X$ at $x \in X$. 
For a differentiable map $c:X \rightarrow Y$ between the differentiable manifolds, let ${dc}_{x}:T_xX \rightarrow T_{c(x)} Y$ be the differential at $x$ and it is a linear map. A point $x \in X$ is a {\it singular point} of $c$ if the rank of the linear map ${dc}_x$ drops. Let $S(c)$ denote the set of all singular points of $c$ and the {\it singular set} of $c$. We only consider smooth maps (maps of the class $C^{\infty}$) as differentiable maps. A {\it diffeomorphism} means a homeomorphism which is smooth and has no singular point. Two smooth manifolds are {\it diffeomorphic} if there exists a diffeomorphism between these manifolds.
\subsubsection{Real algebraic objects.}
We define real algebraic objects respecting existing related classical and sophisticated theory presented in \cite{bochnakcosteroy, kollar, shiota} and our papers and preprints such as \cite{kitazawa2, kitazawa4, kitazawa5}.

A connected component of the zero set of a real polynomial map $c:{\mathbb{R}}^m \rightarrow {\mathbb{R}}^n$ is {\it non-singular} if the rank of $c$ does not drop at any point of $x \in c^{-1}(0) \subset {\mathbb{R}}^m$: the implicit function theorem is respected. A {\it semi-algebraic} set of the real affine space ${\mathbb{R}}^m$ means a subset of ${\mathbb{R}}^m$ represented as the intersection of finitely many sets each of which is either of the form $\{x \in {\mathbb{R}}^m \mid c_j(x)>0\}$ or $\{x \in {\mathbb{R}}^m \mid c_j(x) \geq 0\}$, or $\{x \in {\mathbb{R}}^m \mid  c_j(x)=0\}$, where $c_j$ is a polynomial function. The zero set of a real polynomial map and the set represented as a union of connected components of the zero set of the map is regarded as a semi-algebraic set and called a {\it real algebraic} set. In several articles and preprints, we have called such subsets of the zero sets of real polynomial maps as {\it real algebraic} manifolds if the sets are non-empty and non-singular. We call such a non-singular manifold a {\it regular real algebraic} manifold. This is also a smooth closed manifold. The intersection of finitely many sets each of which is of the form $\{x \in {\mathbb{R}}^m \mid c_j(x)>0\}$ is an $m$-dimensional smooth manifold with no boundary if it is non-empty. We call such a manifold a {\it non-regular} {\it real algebraic} manifold. We call subsets of ${\mathbb{R}}^m$ of these two types {\it real algebraic} manifolds. The real affine space ${\mathbb{R}}^k$ and the unit sphere $S^{k-1} \subset {\mathbb{R}}^{k}$ are regular real algebraic manifolds and ${\mathbb{R}}^k$ and the interior $D^k-S^{k-1} \subset {\mathbb{R}}^k$ of the unit disk $D^k\subset {\mathbb{R}}^k$ is a non-regular real algebraic manifold.

A {\it real algebraic} function $q:S \rightarrow \mathbb{R}$ on a semi-algebraic set $S \subset {\mathbb{R}}^m$ is a smooth function such that ${\Sigma}_{j=1}^{i+1} p_j(x){q(x)}^{j-1}=0$ at any point $x \in S$ for some family $\{p_j(x)\}_{j=1}^{i+1}$ of finitely many real polynomials of $m$ variables with $i$ being a positive integer and $p_{i+1}(x)$ is not a polynomial giving a constant function with the values being $0$. 
Every real polynomial $r(x)$ with $m$ variables gives a real algebraic functions on ${\mathbb{R}}^m$ and its arbitrary semi-algebraic set mapping $x$ to $r(x)$. We can understand this fact by considering $i=1$, $(p_1(x),p_2(x))=(r(x),1)$ and $q(x)=r(x)$.
A {\it real algebraic} map is a map on a semi-algebraic set each of whose component is a real algebraic function. Furthermore, inductively, we can define maps in the following as {\it real algebraic} maps. Real algebraic functions are of course real algebraic maps.
\begin{itemize}
\item If the real affine space ${\mathbb{R}}^n$ of the target of a real algebraic map can be restricted to another semi-algebraic set of ${\mathbb{R}}^n$, then the resulting map is also a real algebraic map. 
\item The composition of real algebraic maps is also a real algebraic map.
\end{itemize}
\subsubsection{Graphs.}

Our {\it graph} means a 1-dimensional connected and compact CW complex. This is also a finite CW complex. We omit rigorous exposition on fundamental notions on (CW) complexes. An {\it edge} of our graph means a $1$-cell of it and a {\it vertex} of it means a $0$-cell of it. As an extended case, here, a circle is also defined as a {\it graph} with exactly one edge, homeomorphic to $S^1$, and no vertex. A {\it vertex set} (an {\it edge set}) of the graph means the set of all edges (resp. vertices) of it.  An {\it isomorphism} between two graphs is a piecewise smooth homeomorphism mapping the vertex set of a graph into that of the other graph. 
\subsection{The Reeb graph of a smooth function into a non-singular curve.}
We can define the {\it Reeb graph} of a function into a non-singular curve as follows. For a smooth function $c:X \rightarrow C$ on a smooth closed manifold $X$ into a non-singular curve $C$ such that the image $c(S(c))$ of the singular set of $c$ is a finite set, consider the following.
\begin{itemize}
\item Let ${\sim}_c$ denote the relation on $X$ as follows: $x_1 {\sim}_c x_2$ if and only if $x_1$ and $x_2$ are in a same connected component of a preimage $c^{-1}(y)$. This is the equivalence relation. The quotient space $W_c:=X/{{\sim}_c}$ is  the {\it Reeb space} of $c$. Let $q_c:X \rightarrow W_c$ denote the quotient map. We also have a continuous function $\bar{c}:W_C \rightarrow \mathbb{R}$ with $c=\bar{c} \circ q_c$ uniquely.
\item \cite{saeki2} guarantees that we have a graph by the following. A point $v$ there is a vertex of the graph if and only if the preimage ${q_c}^{-1}(v)$ contains some singular point of $c$ and this graph is the {\it Reeb graph} of $c$.
\end{itemize} 
Note that Reeb graphs are very classical tools, appearing in \cite{reeb} with Morse functions. They have been also strong tools in theory of Morse functions and applications to geometry of manifolds. We omit precise presentations on related studies.
\subsection{Our main result, stating that for a finite graph equipped with a piecewise smooth function of a certain class satisfying a kind of genericity, and the content of our paper.}
Our result is on explicit reconstruction of a real algebraic function into a non-singular real algebraic curve whose Reeb graph is as prescribed. This extends our main result of \cite{kitazawa2, kitazawa3}. This also respects related similar or extended results of the author such as ones presented in \cite{kitazawa4, kitazawa5, kitazawa6}. 

In the second section, we explicitly exhibit our main result, as Theorem \ref{thm:1}. We also prove this there. The proof is a kind of direct extension of arguments and results in the presented article sand preprints. We also present an explicit case for $C:=S^1$, for example (Theorems \ref{thm:2} and \ref{thm:3}). The third section remarks our result. 
\section{Our main result.}
Let $(a,b):=\{x \in \mathbb{R} \mid a<x<b\}$ and $[a,b]:=\{x \in \mathbb{R} \mid a \leq x \leq b\}$ for two real numbers $a<b$.

A {\it Morse} function $c:X \rightarrow C$ on a manifold $X$ means a smooth function into a $1$-dimensional smooth manifold $C$ with no boundary which has no singular point on the boundary of $X$ and at each singular point $p$ of which
 we have the representation $c(x_1,\cdots x_m)={\Sigma}_{j=1}^{m-i(p)} {x_j}^2-{\Sigma}_{j=1}^{i(p)} {x_{m-i(p)+j}}^2$ for suitable local coordinates and an integer $0 \leq i(p) \leq \frac{m}{2}$. Note that we can define $i(p)$ uniquely and by respecting an orientation of $C$, we can define an integer $0 \leq i(p) \leq m$ uniquely.

{\it Morse-Bott} functions are defined as extensions of Morse functions: they are at each singular point represented as the composition of a smooth map with no singular point with a Morse function. 

For such functions, check \cite{milnor} and see also \cite{bott}, for example.

The {\it degree} of a vertex of a graph means the number of edges containing the vertex.
\begin{Thm}
\label{thm:1}
Let $C$ be a $1$-dimensional connected regular real algebraic manifold in ${\mathbb{R}}^2$ embedded into a $2$-dimensional non-regular real algebraic manifold $N_{\rm Int} (C) \subset {\mathbb{R}}^2$. 
Let $i_C:C \rightarrow N_{\rm Int} (C)$ denote the inclusion.
We assume the following.
\begin{enumerate}
\item \label{thm:1.1} There exists another semi-algebraic set $N(C)$ of ${\mathbb{R}}^2$ and the following are satisfied.
\begin{itemize}
\item The set $N(C)$ is also a smooth, connected and compact manifold diffeomorphic to $C \times [-1,1]$ whose interior considered in ${\mathbb{R}}^2$ is $N_{\rm Int} (C) \subset {\mathbb{R}}^2$ and which admits a diffeomorphism ${\phi}_{C,N(C)}:N(C) \rightarrow C \times [-1,1]$ mapping $x \in C$ to $(x,0) \in C \times \{0\} \subset C \times [-1,1]$ and mapping $N_{\rm Int}(C)$ onto $C \times (-1,1)$. 
\item There also exists a real algebraic map with no singular point ${\pi}_{C,N(C)}:N(C) \rightarrow C$ such that ${\pi}_{C,N(C)} \circ i_C$  is the identity map on $C$.
\item For an arbitrary set $A_C \subset C$, the diffeomorphism ${\phi}_{C,N(C)}$ maps ${{\pi}_{C,N(C)}}^{-1}(A_C)$ onto $A_C \times [-1,1]$.
\end{itemize}

\item \label{thm:1.2}
There exists a piecewise smooth map $c_G:G \rightarrow C$ of a graph $G$ into $C$ satisfying the following.
\begin{itemize}
\item The degree of each vertex of $G$ is $1$ or $3$.
\item The restriction of $c_G$ to each edge is a smooth embedding. For each vertex $v \in G$ of degree $3$ and some small regular neighborhood $N(v)$ of $v$ in the graph $G$, the value $c_G(v)$ is in the interior of the set $c_G(N(v)) \subset C$ considered in the curve $C$. The restriction of $c_G$ to the vertex set of $G$ is also injective.
\item For some piecewise smooth embedding $\tilde{c_G}:G \rightarrow N_{\rm Int} (C) $ and the restriction ${\pi}_{C,N(C)} {\mid}_{N_{\rm Int} (C)}:N_{\rm Int} (C) \rightarrow C$, we have the relation $c_G={\pi}_{C,N(C)} {\mid}_{N_{\rm Int} (C)} \circ \tilde{c_G}$.
\end{itemize}
\end{enumerate}
Let $m \geq 2$ be a positive integer. Then there exist an $m$-dimensional closed and connected manifold $M \subset {\mathbb{R}}^{m+1}$ which is also a regular real algebraic manifold and the zero set of a real polynomial function and a real algebraic map $f_{C,N(C)}:M \rightarrow N_{\rm Int} (C)$ with a function ${\pi}_{C,N(C)} {\mid}_{N_{\rm Int} (C)} \circ f_{C,N(C)}:M \rightarrow C$ which is Morse and whose Reeb graph is isomorphic to $G$.
\end{Thm}
\begin{Rem}
In short, the condition (\ref{thm:1.1}) of Theorem \ref{thm:1} is for the structure of a natural product bundle of $N(C)$ over $C$ whose fiber is $[-1,1]$. For classical and fundamental theory of bundles, see \cite{steenrod} and see also \cite{milnorstasheff} for example. The manifold $N(C)$ is also a closed tubular neighborhood of $C$ and (a kind of specific cases of) this is discussed in \cite[Discussion 7]{kollar}, for example. Later, Example \ref{ex:1} presents explicit cases.

The condition (\ref{thm:1.2}) is for genericity of embedding of the graph and respects the case of \cite{bodinpopescupampusorea}, discussed for the case $C:=\mathbb{R}$.

\end{Rem}
Hereafter, we also need to understand some fundamental notions and arguments on singularity theory, explained in \cite{golubitskyguillemin}.

Related to this, we explain Whitney topologies on spaces of smooth maps from a smooth manifold into another smooth manifold.
For a positive integer $r>0$, the {\it Whitney $C^r$ topology} of the space of all smooth maps from a smooth manifold $X$ into another smooth manifold $Y$ is defined by the following roughly and understanding the rigorous definition is an exercise for readers: two maps are close if and only if their values at each point of $X$ and their $j$-th derivatives at each point of $X$ are close for $1 \leq j \leq r$. The Whitney $C^{r_2}$ topology of the space is stronger than the Whitney $C^{r_1}$ topology of it for $r_2>r_1$. We can also define the {\it Whitney $C^{\infty}$ topology} of it as a stronger topology (as the inductive limit for these topologies).   
\begin{proof}[A proof of Theorem \ref{thm:1}.]
We mainly respect \cite{bodinpopescupampusorea} with our paper \cite{kitazawa2} and a preprint \cite{kitazawa6}.

We can choose a small regular neighborhood $N(G) \subset N_{\rm Int}(C)$ of the graph $\tilde{c_G}(G)$ in the smooth category as presented in \cite{hirsch} for example. We can also consider approximating the boundary of this regular neighborhood by the zero set of some real polynomial function in the $C^r$ or $C^{\infty}$ Whitney topology with $r>1$ and we have a new small regular neighborhood $N_0(G) \subset N_{\rm Int}(C)$ of the graph $\tilde{c_G}(G)$. For this kind of approximation check related surveys presented in \cite{kollar, lellis} for example. This is also used in \cite{bodinpopescupampusorea} and motivated by this we use this in \cite{kitazawa2, kitazawa6} for example. 

We discuss related arguments of the paper \cite{kitazawa2}. We also respect the preprint \cite{kitazawa6} where we do not need to understand this preprint. Let $f_{\tilde{c_G},\mathbb{R}}$ be the real polynomial function whose zero set is thr boundary of $N_0(G)$ and we can regard the region $N_0(G) \subset N_{\rm Int}(C)$ surrounded by this boundary and containing the image $\tilde{c_G}(G)$ as the semi-algebraic set defined by the inequality $f_{\tilde{c_G},\mathbb{R}}(x) \geq 0$. We have the zero set $S_{f_{\tilde{c_G},\mathbb{R}}}:=\{(x,y) \in {\mathbb{R}}^2 \times {\mathbb{R}}^{m-1}={\mathbb{R}}^{m+1} \mid f_{\tilde{c_G},\mathbb{R}}(x)-{||y||}^2=0\}$ of $f_{\tilde{c_G},\mathbb{R}}(x)-{||y||}^2$. This is a regular real algebraic manifold in ${\mathbb{R}}^{m+1}$ and such manifolds and their canonical projections are of certain classes generalizing the unit sphere $S^m \subset {\mathbb{R}}^{m+1}$ and the canonical projection ${\pi}_{m+1,k} {\mid}_{S^m}$ of the unit sphere with $m \geq k \geq 1$. More precisely, here, these canonical projections are regarded as {\it special generic} maps, discussed in \cite{saeki1}, mainly. In \cite{saeki1}, fundamental and explicit theory on their differential topological structures and the topologies and differentiable structures of closed manifolds admitting such maps is discussed.  

We can define our desired map $f_{C,N(C)}:={{\pi}_{m+1,2}} {\mid}_{S_{f_{\tilde{c_G},\mathbb{R}}}}:M:=S_{f_{\tilde{c_G},\mathbb{R}}} \rightarrow N_{\rm Int} (C)$ in such a way that the resulting Reeb graph $W_{{\pi}_{C,N(C)} {\mid}_{N_{\rm Int} (C)} \circ f_{C,N(C)}}$ is isomorphic to $G$ and that the resulting function is a Morse function, by considering the approximation suitably, beforehand. For this, especially, for the graphs, we also respect main arguments of \cite{bodinpopescupampusorea} and generalize the arguments for the conditions (\ref{thm:1.1}, \ref{thm:1.2}). For checking that the functions are Morse and related singularity theory, check \cite{golubitskyguillemin} for example.

This completes the proof.

\end{proof}

Here, we review the definition of a {\it special generic} map. A smooth map $c:X \rightarrow Y$ between smooth manifolds with no boundaries is {\it special generic}, if we have the representation $c(x_1,\cdots x_m)=(x_1,\cdots,x_{n-1},{\Sigma}_{j=1}^{m-n+1} {x_{n-1+j}}^2)$ ($m \geq n \geq 1$) for suitable local coordinates. 

In the following, we present
Example \ref{ex:1}, which is for $N(C)$ in Theorem \ref{thm:1}.
\begin{Ex}
\label{ex:1}
\begin{enumerate}
\item \label{ex:1.1} As a simplest case, we can consider the case $C:=\{(x,0) \mid x \in \mathbb{R}\}$ with a positive real number $a>0$ and $N(C):=\{(x,t) \mid x \in \mathbb{R}, -a \leq t \leq a\}$ with ${\pi}_{C,N(C)}(x,t):=x$.
\item \label{ex:1.2} As another simplest case, we can consider the case $C:=S^1$ with a positive real number $0<a<1$ and $N(C):=\{x \in {\mathbb{R}}^2 \mid 1-a \leq ||x|| \leq 1+a\}$ with ${\pi}_{C,N(C)}(x):=(\frac{1}{||x||}x_1,\frac{1}{||x||}x_2)$ ($x=(x_1,x_2)$).
\end{enumerate}
\end{Ex}
The following is another result related to Example \ref{ex:1} (\ref{ex:1.2}) and Theorem \ref{thm:1}.
\begin{Thm}
\label{thm:2}
Let $G$ be a graph having exactly $i+1$ vertices with $i+1>2$. Let $\{v_j\}_{j=1}^{i+1}$ be the family of the $i+1$ vertices.
Each of the closures of edges of $G$ connects $v_j$ and $v_{j+1}$ for some $1 \leq j \leq i$ or $v_{i+1}$ and $v_1$.

The number of all edges whose closures connect $v_j$ and $v_{j+1}$ is $a_j>0$ and that of all edges whose closures connect $v_{i+1}$ and $v_1$ is $a_{i+1}>0$, satisfying the relation $(a_{j}, a_{j+1}) \neq (1,1)$ for $1 \leq j \leq i$ and $(a_{i+1},a_1) \neq (1,1)$.

Then for any integer $m$ greater than $1$, there exist an $m$-dimensional closed and connected manifold $M \subset {\mathbb{R}}^{m+1}$ which is also a regular real algebraic manifold and the zero set of a real polynomial function of degree $2{\Sigma}_{j=1}^{i+1} (a_j-1)+4$ and a real algebraic map $f_{C,N(C)}:M \rightarrow N(C) \subset {\mathbb{R}}^2$ with a function ${\pi}_{C,N(C)} {\mid}_{N(C)} \circ f_{C,N(C)}:M \rightarrow C$ which is Morse and whose Reeb graph is isomorphic to $G$ where the notation and the situation of Example \ref{ex:1} {\rm (}\ref{ex:1.2}{\rm )} are considered.
\end{Thm}
Hereafter, two $1$-dimensional real algebraic manifolds in ${\mathbb{R}}^2$ are {\it mutually tangent} at a point $p \in {\mathbb{R}}^2$ if they contain $p$ and their tangent vector spaces at $p$ agree.
A {\it circle} of a fixed radius $r>0$ means a real algebraic manifold of the form $\{(x_1,x_2) \mid 
{||x-b||}^2=r\}$ for $x=(x_1,x_2) \in {\mathbb{R}}^2$ and $b=(b_1,b_2) \in {\mathbb{R}}^2$, diffeomorphic to $S^1$.

We also expect readers to know elementary notions and arguments on plane geometry (Euclidean geometry).
\begin{proof}[A proof of Theorem \ref{thm:2}]

To each vertex $v_j$, we correspond $(\cos \frac{2j\pi}{i+1},\sin \frac{2j\pi}{i+1}) \in {\mathbb{R}}^2$.
We can choose a suitable small real number $0<a<1$ and the following mutually disjoint circles $C_a$ each of which bounds the compact disk $D_{C_a} \subset N_{\rm Int} (C)$ with these disks $D_{C_a}$ being mutually disjoint.
\begin{itemize}
\item For each integer $1 \leq j \leq i+1$, exactly $a_j-1$ circles of suitable radii contained in the sector formed and surrounded by $\{(r\cos \frac{2j\pi}{i+1},r\sin \frac{2j\pi}{i+1}) \in {\mathbb{R}}^2 \mid r \geq 0\}$ and $\{(r\cos \frac{2(j+1)\pi}{i+1},r\sin \frac{2(j+1)\pi}{i+1})  \in {\mathbb{R}}^2 \mid r \geq 0\}$ are chosen.
\item Each of the $a_j-1$ circles and $\{(t\cos \frac{2j\pi}{i+1},t\sin \frac{2j\pi}{i+1}) \in {\mathbb{R}}^2 \mid t \in \mathbb{R} \}$ are mutually tangent at a point and the circle and $\{(t\cos \frac{2(j+1)\pi}{i+1},t\sin \frac{2(j+1)\pi}{i+1}) \in {\mathbb{R}}^2  \mid t \in \mathbb{R} \}$ are mutually tangent at another point.
\end{itemize}

By removing the interiors of the disks $D_{C_a}$ from $N(C)$, we have a compact and connected region surrounded by circles, in $N(C)$. 
The union of the circles is regarded as the zero set of a product of $2+{\Sigma}_{j=1}^{i+1}(a_j-1)$ functions of the form ${||x-b||}^2-r$ with $b=(b_1,b_2) \in {\mathbb{R}}^2$, $r>0$ and the variable $x=(x_1,x_2)$.
The region, which is a compact and connected set in ${\mathbb{R}}^2$, is regarded as a semi-algebraic set in ${\mathbb{R}}^2$ and we can have a real algebraic map onto the resulting region like the map presented in the proof of Theorem \ref{thm:1}, according to \cite{kitazawa2} (\cite{kitazawa6}). 
By the argument in the statement, we have our desired result. Note that the conditions on the values $a_j$ are to construct our function with some singular points in each preimage ${f_{C,N(C)}}^{-1}(\{(r\cos \frac{2j\pi}{i+1},r\sin \frac{2j\pi}{i+1}) \in {\mathbb{R}}^2 \mid r \geq 0\})$ for each integer $0 \leq j \leq i+1$.

This completes the proof. 
\end{proof}

Note that some graphs of Theorem \ref{thm:2} may not satisfy the conditions of Theorem \ref{thm:1}.
Note again that Theorems \ref{thm:1} and \ref{thm:2} are extended from the corresponding cases for $C:=\mathbb{R}$ in \cite{kitazawa2, kitazawa3, kitazawa4, kitazawa6} for example.

We present Theorem \ref{thm:3}, as another result. 

Hereafter, we need the notion of connected sum and boundary connected sum of manifolds. We consider this in the differentiable (smooth) category.

We use ${\sharp}_{j=1}^{l_1} M_j$ for a connected sum of these $l_1$ connected manifolds $M_j$. We use ${\natural}_{j=1}^{l_2} N_j$ for a boundary connected sum of these $l_2$ connected manifolds $N_j$. These smooth manifolds are defined uniquely up to (the existence of) diffeomorphisms and this rule contains no problem. 

Hereafter, an {\it ellipsoid} means a semi-algebraic set of the form $\{x \in {\mathbb{R}}^k \mid a_j{x_j}^2 \leq r \}$ with $x_j$ being the $j$-th component of $x$, $a_j>0$ and $r>0$. It is diffeomorphic to $D^k$.

We introduce a method to construct regular real algebraic manifolds from semi-algebraic sets of a certain class. Such construction is presented in our proof of Theorem \ref{thm:1} and \ref{thm:2} and this also reviews the construction.

\begin{Def}
\label{def:1}
Let $F$ be a real polynomial function with $k>0$ variables. Let $k^{\prime}$ be a positive integer.
According to the paper \cite{kitazawa2}, followed by \cite{kitazawa6}, from a semi-algebraic set $\{x \in {\mathbb{R}}^k \mid F(x) \geq 0\}$ which is connected and which is surrounded by the zero set $\{x \in {\mathbb{R}}^k \mid F(x)=0\}$ being also a regular algebraic manifold of ${\mathbb{R}}^{k}$, we have a {\rm (}$k+k^{\prime}-1${\rm )}-dimensional regular algebraic manifold of ${\mathbb{R}}^{k+k^{\prime}}$ which is connected, represented as $\{(x,y) \in {\mathbb{R}}^{k+k^{\prime}} \mid F(x)-{||y||}^2=0\}$, and also the zero set of $F(x)-{||y||}^2$. We can also consider the semi-algebraic set $\{(x,y) \in {\mathbb{R}}^{k+k^{\prime}} \mid F(x)-{||y||}^2 \geq 0\}$ being also connected. 

We call this method a {\it unit-sphere-construction} or a {\it US-construction}.

\end{Def}

In Definition \ref{def:1}, the restriction of ${\pi}_{k+k^{\prime},k}$ to $\{(x,y) \in {\mathbb{R}}^{k+k^{\prime}} \mid F(x)-{||y||}^2=0\}$ is a special generic map thanks to some discussion from \cite{kollar} such as \cite[Discussion 14]{kollar}, with \cite{kitazawa2}.

Remember the definitions of edges and vertices of a graph. Each edge is, by definition, an open set of the graph. Hereafter, the closure of an edge of the graph is chosen in the graph.

\begin{Thm}
\label{thm:3}
Let $m>2$ be an integer. Let $m^{\prime} \geq 1$ be the maximal integer satisfying $m^{\prime} \leq \frac{m-1}{2}$.

Let $G$ be a graph having exactly $i+1$ vertices with $i+1>2$. Let $\{v_j\}_{j=1}^{i+1}$ be the family of the $i+1$ vertices.
Each of the closures of edges of $G$ connects $v_j$ and $v_{j+1}$ for some $1 \leq j \leq i$ or $v_{i+1}$ and $v_1$. There exist $a_j>0$ edges whose closure connects $v_j$ and $v_{j+1}$ for each $1 \leq j \leq i$, and $a_{i+1}>0$ edges whose closure connects $v_{i+1}$ and $v_{1}$.

To each edge $e_{j,j^{\prime}}$ in the family $\{e_{j,j^{\prime}}\}_{j^{\prime=1}}^{a_j}$ the closure of each of which connects $v_j$ and $v_{j+1}$ {\rm (}$1 \leq j \leq i${\rm )} and each edge $e_{i+1,j^{\prime}}$ in the family $\{e_{i+1,j^{\prime}}\}_{j^{\prime=1}}^{a_{i+1}}$ the closure of each of which connects $v_{i+1}$ and $v_{1}$, a sequence $\{a_{e_{j,j^{\prime}},j^{\prime \prime}}\}_{j^{\prime \prime}=1}^{m^{\prime}}$ of non-negative integers of length $m^{\prime}$ is assigned obeying the following rule{\rm :} if $a_j=a_{j+1}=1$, then it does not hold that $a_{e_{j,1,j^{\prime \prime}}}=a_{e_{j+1,1,j^{\prime \prime}}}=0$ for any $1 \leq j^{\prime \prime} \leq m^{\prime}$, and if $a_{i+1}=a_{1}=1$, then it does not hold that $a_{e_{i+1,1,j^{\prime \prime}}}=a_{e_{1,1,j^{\prime \prime}}}=0$ for any $1 \leq j^{\prime \prime} \leq m^{\prime}$.

Then, there exist an $m$-dimensional closed and connected manifold $M \subset {\mathbb{R}}^{m+1}$ which is also a regular real algebraic manifold and the zero set of a real polynomial function of degree $2{\Sigma}_{j=1}^{i+1} ({\Sigma}_{j^{\prime}=1}^{a_j} ({\Sigma}_{j^{\prime \prime}=1}^{m^{\prime \prime}} a_{e_{j,j^{\prime},j^{\prime \prime}}}))+2{\Sigma}_{j_1=1}^{i+1}(a_{j_1}-1)+4$ and a real algebraic map $f_{C,N(C)}:M \rightarrow N(C)$ with a function ${\pi}_{C,N(C)} {\mid}_{N(C)} \circ f_{C,N(C)}:M \rightarrow C$ enjoying the following where the notation and the situation of Example \ref{ex:1} {\rm (}\ref{ex:1.2}{\rm )} are considered.
\begin{enumerate}
\item The function is Morse.
\item The Reeb graph of the function is isomorphic to $G$.
\item The preimage ${q_{{\pi}_{C,N(C)} {\mid}_{N(C)} \circ f_{C,N(C)}}}^{-1}(p)$ {\rm (}$p \in e_{j,j^{\prime}}${\rm )} is diffeomorphic to a manifold represented as a connected sum ${\sharp}_{j^{\prime \prime}=1}^{m^{\prime}} {\sharp}_{j^{\prime \prime \prime}=1}^{a_{e_{j,j^{\prime}},j^{\prime \prime}}} (S^{j^{\prime \prime}} \times S^{m-j^{\prime \prime}-1})$.
\end{enumerate}
\end{Thm}
\begin{proof}
First, this respects our main result of \cite{kitazawa4} and its proof, which is a variant in the case "$C:=\mathbb{R}$" instead of "$C:=S^1$". We do not assume the arguments there.

As in the proof of Theorem \ref{thm:2}, to each vertex $v_j$, we correspond $(\cos \frac{2j\pi}{i+1},\sin \frac{2j\pi}{i+1}) \in {\mathbb{R}}^2$.
We can choose a suitable small real number $0<a<1$ and the following mutually disjoint circles $C_{a,j_1,j_2}$ and $C_{a,\{j_i\}_{i=1}^{3},j_4}$ each of which bounds the compact disk $D_{C_{a,j_1,j_2}} \subset N_{\rm Int} (C)$ with these disks being mutually disjoint and the compact disk $D_{C_{a,\{j_i\}_{i=1}^{3},j_4}} \subset N_{\rm Int} (C)$ with these disks being mutually disjoint and disjoint from the disks $D_{C_{a,j_1,j_2}}$.
\begin{itemize}
\item For each integer $1 \leq j_1 \leq i+1$, we can choose exactly $a_{j_1}-1$ circles $C_{a,j_1,j_2}$ of suitable radii each of which is labeled by an integer $1 \leq j_2 \leq a_{j_1}-1$ and contained in the sector formed and surrounded by $\{(r\cos \frac{2j_1\pi}{i+1},r\sin \frac{2j_1\pi}{i+1}) \in {\mathbb{R}}^2 \mid r \geq 0\}$ and $\{(r\cos \frac{2(j_1+1)\pi}{i+1},r\sin \frac{2(j_1+1)\pi}{i+1})  \in {\mathbb{R}}^2 \mid r \geq 0\}$ in such a way that the circle and $\{(t\cos \frac{2j_1\pi}{i+1},t\sin \frac{2j_1\pi}{i+1}) \in {\mathbb{R}}^2 \mid t \in \mathbb{R} \}$ are mutually tangent at a point and that the circle and $\{(t\cos \frac{2(j_1+1)\pi}{i+1},t\sin \frac{2(j_1+1)\pi}{i+1}) \in {\mathbb{R}}^2  \mid t \in \mathbb{R} \}$ are mutually tangent at another point.
\item For each integer $1 \leq j_1 \leq i+1$ and each integer $1 \leq j_2 \leq a_{j_1}$, we have exactly ${\Sigma}_{j_3=1}^{m^{\prime}} a_{e_{j_1,j_2},j_3}$ circles $C_{a,\{j_i\}_{i=1}^{3},j_4}$ of suitable radii each of which is labeled by an integer $1 \leq j_4 \leq a_{e_{j_1,j_2},j_3}$ with each integer $1 \leq j_3 \leq m^{\prime}$ and contained in the bounded region explained in the following. \\
\ \\
Case A. $j_2=1<a_{j_1}$. \\
\begin{itemize}
\item $\{(r\cos \frac{2j_1\pi}{i+1},r\sin \frac{2j_1\pi}{i+1}) \in {\mathbb{R}}^2 \mid r \geq 0\}$.
\item $\{(r\cos \frac{2(j_1+1)\pi}{i+1},r\sin \frac{2(j_1+1)\pi}{i+1})  \in {\mathbb{R}}^2 \mid r \geq 0\}$.
\item $\{x \in {\mathbb{R}}^2 \mid ||x||=1-a\}$.
\item $C_{a,j_1,1}$.
\end{itemize}
\ \\
Case B. $1<j_2<a_{j_1}$. \\
\begin{itemize}
\item $\{(r\cos \frac{2j_1\pi}{i+1},r\sin \frac{2j_1\pi}{i+1}) \in {\mathbb{R}}^2 \mid r \geq 0\}$.
\item $\{(r\cos \frac{2(j_1+1)\pi}{i+1},r\sin \frac{2(j_1+1)\pi}{i+1})  \in {\mathbb{R}}^2 \mid r \geq 0\}$.
\item $C_{a,j_1,j_2-1}$.
\item $C_{a,j_1,j_2}$.
\end{itemize}
\ \\
Case C. $1<j_2=a_{j_1}$. \\
\begin{itemize}
\item $\{(r\cos \frac{2j_1\pi}{i+1},r\sin \frac{2j_1\pi}{i+1}) \in {\mathbb{R}}^2 \mid r \geq 0\}$.
\item $\{(r\cos \frac{2(j_1+1)\pi}{i+1},r\sin \frac{2(j_1+1)\pi}{i+1})  \in {\mathbb{R}}^2 \mid r \geq 0\}$.
\item $\{x \in {\mathbb{R}}^2 \mid ||x||=1+a\}$.
\item $C_{a,j_1,j_2-1}$.
\end{itemize}
\ \\
Case D. $j_2=1=a_{j_1}$. \\
\begin{itemize}
\item $\{(r\cos \frac{2j_1\pi}{i+1},r\sin \frac{2j_1\pi}{i+1}) \in {\mathbb{R}}^2 \mid r \geq 0\}$.
\item $\{(r\cos \frac{2(j_1+1)\pi}{i+1},r\sin \frac{2(j_1+1)\pi}{i+1})  \in {\mathbb{R}}^2 \mid r \geq 0\}$.
\item $\{x \in {\mathbb{R}}^2 \mid ||x||=1-a\}$.

\item $\{x \in {\mathbb{R}}^2 \mid ||x||=1+a\}$.

\end{itemize}
\ \\
\item Each circle $C_{a,\{j_i\}_{i=1}^{3},j_4}$ and $\{(t\cos \frac{2j_1\pi}{i+1},t\sin \frac{2j_1\pi}{i+1}) \in {\mathbb{R}}^2 \mid t \in \mathbb{R} \}$ are mutually tangent at a point and this circle and $\{(t\cos \frac{2(j_1+1)\pi}{i+1},t\sin \frac{2(j_1+1)\pi}{i+1}) \in {\mathbb{R}}^2  \mid t \in \mathbb{R} \}$ are mutually tangent at another point.
\end{itemize}
We can discuss our desired construction in the following way. 
\begin{itemize}
\item We remove the interiors of the disks $D_{C_{a,j_1,j_2}}$ from $N(C)$ and for the resulting connected semi-algebraic set in ${\mathbb{R}}^2$, we can consider the US construction as in Definition \ref{def:1} by $(k,k^{\prime})=(2,1)$ with the degree of $F$ being $2{\Sigma}_{j_1=1}^{i+1}(a_{j_1}-1)+4$. Let $D_C \subset {\mathbb{R}}^3$ denote the resulting semi-algebraic set being also compact and connected.
\item After that, we discuss our construction inductively for each integer $1 \leq j_3 \leq m^{\prime}$ as follows.
First we put $n=3$ and $j_3=1$. 
\begin{itemize}
\item We choose a suitable ellipsoid $E_{a,\{j_i\}_{i=1}^{3},j_4}$ embedded in the interior of $D_C$ considered in ${\mathbb{R}}^n$ and mapped onto each disk $D_{a,\{j_i\}_{i=1}^{3},j_4}$ by the projection ${\pi}_{n,2}$, for each $j_3$ from $D_C$.
\item \begin{itemize}
\item If $j_3<m^{\prime}$, then for the resulting connected semi-algebraic set in ${\mathbb{R}}^n$, we can consider the US construction of Definition \ref{def:1} with $(k,k^{\prime})=(n,1)$ with the degree of $F$ being $2{\Sigma}_{j=1}^{i+1} ({\Sigma}_{j^{\prime}=1}^{a_j} ({\Sigma}_{j^{\prime \prime}=1}^{j_3} a_{e_{j,j^{\prime},j^{\prime \prime}}}))+2{\Sigma}_{j_1=1}^{i+1}(a_{j_1}-1)+4$. We put $D_C \subset {\mathbb{R}}^{n+1}$ the resulting new semi-algebraic set being also compact and connected newly instead of "the previously defined $D_C$".
We define $n$ as $n+1$ newly, instead. We define $j_3$ as $j_3+1$ newly, instead. 
\item If $j_3=m^{\prime}$, then for the resulting connected semi-algebraic set in ${\mathbb{R}}^n$, we can consider the US construction of Definition \ref{def:1} with $(k,k^{\prime})=(n,m-n+1)$ with the degree of $F$ being $2{\Sigma}_{j=1}^{i+1} ({\Sigma}_{j^{\prime}=1}^{a_j} ({\Sigma}_{j^{\prime \prime}=1}^{m^{\prime \prime}} a_{e_{j,j^{\prime},j^{\prime \prime}}}))+2{\Sigma}_{j_1=1}^{i+1}(a_{j_1}-1)+4$. Let $W \subset {\mathbb{R}}^{m+1}$ denote the resulting new semi-algebraic set being also compact and connected newly and $M$ denote the boundary of the region $W$.
\end{itemize}
\end{itemize}
\end{itemize}
Thus we have our desired map $f:M \rightarrow N(C) \subset {\mathbb{R}}^2$ by ${\pi}_{n,2} {\mid}_{D_C} \circ {\pi}_{m+1,n} {\mid}_{M}={\pi}_{m+1,2} {\mid}_{M}$. We explain some more precisely.

We explain the preimage ${q_{{\pi}_{C,N(C)} {\mid}_{N(C)} \circ f_{C,N(C)}}}^{-1}(p)$ {\rm (}$p \in e_{j,j^{\prime}}${\rm )}. The image ${\pi}_{m+1,n}({q_{{\pi}_{C,N(C)} {\mid}_{N(C)} \circ f_{C,N(C)}}}^{-1}(p))$ is regarded as a manifold diffeomorphic to the boundary connected sum ${\natural}_{j^{\prime \prime}=1}^{m^{\prime}} {\natural}_{j^{\prime \prime \prime}=1}^{a_{e_{j,j^{\prime}},j^{\prime \prime}}} (S^{j^{\prime \prime}} \times D^{n-j^{\prime \prime}-1})$. The restriction of ${\pi}_{m+1,n}$ to the preimage ${q_{{\pi}_{C,N(C)} {\mid}_{N(C)} \circ f_{C,N(C)}}}^{-1}(p)$ {\rm (}$p \in e_{j,j^{\prime}}${\rm )} is regarded as a special generic map into the preimage ${{\pi}_{n,2}}^{-1}(\{(r\cos \frac{2(j+\theta)\pi}{i+1},r\sin \frac{2(j+\theta)\pi}{i+1}) \in {\mathbb{R}}^2 \mid r>0\})$ of some straight open interval $\{(r\cos \frac{2(j+\theta)\pi}{i+1},r\sin \frac{2(j+\theta)\pi}{i+1}) \in {\mathbb{R}}^2 \mid r>0\}$ with $0<\theta<1$. We can easily see that ${{\pi}_{n,2}}^{-1}(\{(r\cos \frac{2(j+\theta)\pi}{i+1},r\sin \frac{2(j+\theta)\pi}{i+1}) \in {\mathbb{R}}^2 \mid r>0\})$ is diffeomorphic to ${\mathbb{R}}^{n-1}$. We can also understand the type (the topology and the differentiable structure) of the preimage  ${q_{{\pi}_{C,N(C)} {\mid}_{N(C)} \circ f_{C,N(C)}}}^{-1}(p)$ {\rm (}$p \in e_{j,j^{\prime}}${\rm )}.
We can investigate this as a kind of exrecises on singularity theory related to differential topology or special generic maps. Checking \cite{saeki1} with \cite{kitazawa7} may also help us to understand this, where we do not assume mathematical knowledge and experience related to the preprint \cite{kitazawa7}.

Note also that the conditions on the values $a_{e_{j,j^{\prime}},j^{\prime \prime}}$ are to construct our function with some singular points in each preimage ${f_{C,N(C)}}^{-1}(\{(r\cos \frac{2j\pi}{i+1},r\sin \frac{2j\pi}{i+1}) \in {\mathbb{R}}^2 \mid r \geq 0\})$ for each integer $0 \leq j \leq i+1$.

This completes our proof.

\end{proof}
\section{Additional comments.}
This is a kind of our additional remark.
 
Hereafter, rigorously, {\it real-valued} functions mean functions whose values are always real numbers or elements of $\mathbb{R}$. {\it Circle-valued} functions mean functions whose values are always elements of $S^1$. {\it Curve-valued} functions have been used for real algebraic maps into $1$-dimensional real algebraic manifolds and we use this terminology.

Most of the comments are related to Morse(-Bott) functions on compact manifolds, mainly ones on compact surfaces.

This paper studies extensions of several affirmative answers on reconstructing real algebraic real-valued Morse functions to curve-valued cases.

We first introduce related studies on differentiable (smooth) real-valued Morse(-Bott) functions and extensions to the circle-valued cases. 
\subsection{From real-valued Morse(-Bott) functions and their Reeb graphs in the differentiable situations to circle-valued versions.} 

\subsubsection{Circle-valued Morse functions and their Reeb graphs.}
In the differentiable (smooth) situations, it is important to extend some affirmative facts on reconstruction of nice smooth real-valued functions to circle-valued cases or find variants of facts on the real-valued function case.
For example, Theorem \ref{thm:3} holds.
\begin{Thm}[\cite{gelbukh1}]
\label{thm:4}
Every graph is homeomorphic to the Reeb graph of some real-valued Morse-Bott function on some closed and connected surface and also homeomorphic to the Reeb graph of some real-valued Morse-Bott function on some closed and connected manifold of dimension $m>2$ where $m>2$ is an arbitrary integer greater than $2$.
\end{Thm}
For example, if the degrees of vertices of the graph are at least $3$, then the graph is not isomorphic to the Reeb graph of any Morse-Bott function on any closed and connected surface. On the other hand, we have the following in the circle-valued case.
\begin{Thm}[\cite{gelbukh2, gelbukh3}]
\label{thm:5}
Every graph is isomorphic to the Reeb graph of some circle-valued Morse-Bott function on some closed and connected surface and also isomorphic to the Reeb graph of some real-valued Morse-Bott function on some closed and connected manifold of dimension $m>2$ where $m>2$ is an arbitrary integer greater than $2$.
\end{Thm}
\subsubsection{Non-singular extensions of real-valued Morse functions on closed surfaces and extensions to the case of circle-valued ones.}
We introduce another study related to real-valued Morse functions and their Reeb graphs and extensions to the circle-valued cases, shortly.
{\it Non-singular extensions} of Morse functions are smooth functions with no singular points on some compact manifolds whose restrictions to the boundaries are the original functions. See \cite{curley} for related studies, for example. 
The existence of such functions is a kind of fundamental and important problems in singularity theory related to differential topology of manifolds. In the case of real-valued Morse functions on closed surfaces, a necessary and sufficient condition to solve the problem affirmatively is given in terms of Reeb graphs of the functions. \cite{iwakura} extends this to the circle-valued case in a certain natural way, by certain additional investigations.

We omit precise discussions on these studies and these results. We only present the following fact, which is easily shown.
\begin{Thm}
\label{thm:6}
For functions in Theorems \ref{thm:1} and \ref{thm:2}, we can have their non-singular extensions being also real algebraic maps on some connected and compact manifolds being also regarded as semi-algebraic sets of ${\mathbb{R}}^{m+1}$.

\end{Thm}
\begin{proof}
We only extend the function canonically by considering the semi-algebraic set $\{(x,y) \in {\mathbb{R}}^2 \times {\mathbb{R}}^{m-1}={\mathbb{R}}^{m+1} \mid f_{\tilde{c_G},\mathbb{R}}(x)-{||y||}^2 \geq 0\}$. The desired function is defined as the restriction of ${\pi}_{C,N(C)} {\mid}_{N(C)} \circ {\pi}_{m+1,2}$ there. This completes the proof.
\end{proof}
\subsection{Normal forms of smooth functions, mainly, ones on closed surfaces.}
\cite{feshchenko1, feshchenko2} studies a kind of normal forms of explicit Morse-Bott functions on closed and orientable surfaces which always have local extrema at their singular points or generalizations with respect to the local forms of the singular points of these Morse-Bott functions. 
More precisely, he has considered {\it the class ${\mathcal{F}}^0$} of smooth functions on closed and connected surfaces as follows. A smooth function $c:X \rightarrow C$ on a closed and orientable surface $X$ is of this class if and only if the following hold.
\begin{itemize}
\item The singular set $S(c)$ of $c$ is a disjoint union of isolated points and copies of $S^1$.
\item At each isolated singular point of $c$, it is represented by the form $c(x_1,x_2)=\pm ({x_1}^2+{x_2}^2)$ for suitable local coordinates.
\item At each circle $S_C$ of $S(c)$, it is represented by the form $c(x_1,x_2)=\pm ({x_1}^{n_{S_C}})$ for suitable local coordinates and a suitably chosen positive integer $n_{S_C} \geq 2$ depending on the component $S_C$.
\end{itemize}

In terms of our discussions, we can state some of the result of these studies of Feshchenko as follows.
\begin{Thm}[\cite{feshchenko1, feshchenko2}]
\label{thm:7}
Let $G$ be a connected graph with exactly one edge and two vertices, a circle {\rm (}, which has no vertex{\rm )}, or a graph which is homeomorphic to $S^1$ and has at least two edges and at least two vertices.
\begin{enumerate}
\item We can reconstruct a smooth function $f_{{\mathcal{F}}^0}$ onto some $1$-dimensional connected real algebraic manifold $C$ of the class ${\mathcal{F}}^0$ on some closed, connected and orientable surface $M$ whose Reeb graph is isomorphic to $G$.
\item We have the following.
\begin{enumerate}
\item If $G$ is a connected graph with exactly one edge and two vertices and a smooth function $f_{{\mathcal{F}}^0}:M \rightarrow C$ of the class ${\mathcal{F}}^0$ is reconstructed in the previous scene, then by choosing a suitable pair $(\Phi:S^2 \rightarrow M,\phi:C \rightarrow \mathbb{R})$ of diffeomorphisms, we have the canonical projection ${\pi}_{3,1} {\mid}_{S^2}$ as $\phi \circ f_{{\mathcal{F}}^0} \circ \Phi:S^2 \rightarrow \mathbb{R}$.
\item If $G$ is a circle {\rm (}, which has no vertex{\rm )} and a smooth function $f_{{\mathcal{F}}^0}:M \rightarrow C$ of the class ${\mathcal{F}}^0$ is reconstructed in the previous scene, then by choosing a suitable pair $(\Phi:S^1 \times S^1 \rightarrow M,\phi:C \rightarrow S^1)$ of diffeomorphisms, we have the projection ${\rm Pr}_{S^1 \times S^1,1}:S^1 \times S^1 \rightarrow S^1$ mapping $(x_1,x_2) \in S^1 \times S^1$ to $x_1 \in S^1$ as $\phi \circ f_{{\mathcal{F}}^0} \circ \Phi: S^1 \times S^1 \rightarrow S^1$.
\item Let $G$ be a graph which is homeomorphic to $S^1$ and has at least two edges and at least two vertices. In this case, if a real-valued smooth function $f_{{\mathcal{F}}^0}:M \rightarrow \mathbb{R}$ of the class ${\mathcal{F}}^0$ is reconstructed in the previous scene, then for a suitable diffeomorphism ${\Phi}:S^1 \times S^1 \rightarrow M$ with a suitable smooth real-valued function $c_{S^1}$ on $S^1$ having finitely many singular points, we have the relation $f_{{\mathcal{F}}^0} =c_{S^1} \circ {\rm Pr}_{S^1 \times S^1,1} \circ {\Phi}^{-1}$. In addition, here, if a circle-valued smooth function $f_{{\mathcal{F}}^0}:M \rightarrow S^1$ of the class ${\mathcal{F}}^0$ is reconstructed in the previous scene, then for a suitable diffeomorphism $\Psi:S^1 \times S^1 \rightarrow M$ with a suitable smooth circle-valued function $c_{S^1}$ on $S^1$ having finitely many singular points, we have the relation $f_{{\mathcal{F}}^0} =c_{S^1} \circ {\rm Pr}_{S^1 \times S^1,1} \circ {\Psi}^{-1}$. 
\end{enumerate}
\end{enumerate}
\end{Thm}
\begin{Rem}
	This is on Theorem \ref{thm:7}.
	
	 For example, in our arguments, we may replace the projection ${\rm Pr}_{S^1 \times S^1,1}:S^1 \times S^1 \rightarrow S^1$ mapping $(x_1,x_2) \in S^1 \times S^1$ to $x_1 \in S^1$ by a real algebraic map on $S^1 \times S^1$ onto $N(C):=\{x \in {\mathbb{R}}^2 \mid \frac{1}{2} \leq ||x|| \leq \frac{3}{2} \}$ of Example \ref{ex:1} (\ref{ex:1.2}) reconstructed as in the proof of Theorem \ref{thm:1} or (\cite{kitazawa2, kitazawa6}).

Our present study is regarded as a kind of generalizations of this study from the viewpoint of singularity theory with topological theory and combinatorial one of real algebraic functions. More precisely, we consider a general graph $G$ instead and investigate variants of the presented result.
\end{Rem}
\section{Conflict of interest and data.}
The author is a researcher at Osaka Central
Advanced Mathematical Institute (OCAMI researcher) and this is supported by MEXT Promotion of Distinctive Joint Research Center Program JPMXP0723833165. The author is not employed there. However, our study thanks this. \\

For the present paper, no other data are not generated.

\end{document}